\documentclass[12pt]{article}
\usepackage[T1]{fontenc}
\usepackage[utf8]{inputenc}
\usepackage[english]{babel}

\usepackage{lmodern}

\usepackage[style=ieee, backend=bibtex]{biblatex}
\usepackage{balance}
\usepackage{microtype}
\usepackage{mathtools}
\usepackage{amssymb}
\usepackage{bm}

\usepackage{amsthm}

\newcommand{\bmat}[1]{ \begin{bmatrix} #1 \end{bmatrix} }
\DeclarePairedDelimiter{\norm}{\lVert}{\rVert}

\renewcommand{\Re}{\mathrm{Re}}

\newcommand{\tran}{{\mkern-1.5mu\mathsf{T}}}

\newcommand{\fieldsym}[1]{{\mbox{\bf #1}}}

\newcommand{\HermH}{\fieldsym{H}}
\newcommand{\HermMats}[1]{\HermH^#1}
\newcommand{\HermPSDCone}[1]{\HermH^#1_+}
\newcommand{\HermPDCone}[1]{\HermH^#1_{\small ++}}

\newcommand{\SymS}{\fieldsym{S}}
\newcommand{\SymMats}[1]{\SymS^#1}
\newcommand{\PSDCone}[1]{\SymS^#1_+}
\newcommand{\PDCone}[1]{\SymS^#1_{\small ++}}

\newcommand{\Span}{\text{span}}
\newcommand{\Pset}{\mathcal{P}}
\newcommand{\LowRank}{\mathcal{R}}

\newcommand{\Reals}{\fieldsym{R}}
\newcommand{\Complexs}{\fieldsym{C}}

\newcommand*\oline[1]{%
	\vbox{%
		\hrule height 0.8pt
		\kern0.35ex
		\hbox{%
			\kern-0.1em
			\ifmmode#1\else\ensuremath{#1}\fi
			\kern-0.1em
		}
	}
}
\newcommand{\sigmamax}{\oline{\sigma}}

\newcommand{\bDelta}{{\bm{\Delta}}}
\newcommand{\im}{\text{im}}
\newcommand{\re}{\text{re}}

\renewcommand{\qed}{\hfill $\square$}

\renewenvironment{proof}{\vspace{-0.35em}\noindent\emph{Proof.}}{\qed}

\theoremstyle{definition}
\newtheorem{theorem}{Theorem}

\newtheorem{definition}{Definition}

\newtheorem{counterexample}{Counterexample}

\newtheorem{manualtheoreminner}{Theorem}
\newenvironment{manualtheorem}[1]{%
	\renewcommand\themanualtheoreminner{#1}%
	\manualtheoreminner
}{\endmanualtheoreminner}

\newtheoremstyle{myremark}
{3pt}
{3pt}
{}
{}
{\bf}
{.}
{.5em}
{}
\theoremstyle{myremark}
\newtheorem*{remark*}{Remark}

\begin{document}

\title{Five-Full-Block Structured Singular Values of\\Real Matrices Equal Their Upper Bounds}%

\author{Olof Troeng%
\thanks{The author is with the Department of Automatic Control, Lund University, Sweden, and is a member of the ELLIIT Strategic Research Area. E-mail: \texttt{oloft@control.lth.se}.}
}

\maketitle


\begin{abstract}
We show that the structured singular value of a real matrix with respect to five full complex uncertainty blocks equals its convex upper bound.
This is done by formulating the equality conditions as a feasibility SDP and invoking a result on the existence of a low-rank solution.
A counterexample is given for the case of six uncertainty blocks. Known results are also revisited using the proposed approach.
\end{abstract}

\section{Introduction}
The structured singular value is the distance of a matrix to singularity with respect to a given class of perturbations \cite{Doyle1982,Packard1993}. It is used for quantifying performance and robustness of dynamical systems subject to structured uncertainty.
The structured singular value is notoriously hard to compute but an upper bound can be found by convex optimization.
For certain uncertainty structures, the convex upper bound equals the structured singular; a list of such structures is given in  \cite[Sec.~9]{Packard1993}.

In this letter we show that in the special case of real matrices, equality holds for a larger class of uncertainty structures than previously known---for up to five full complex uncertainty blocks.
The proof is based on a result by Barvinok \cite{Barvinok1995} that guarantees the existence of low-rank solutions to feasibility SDPs with few constraints.
A counterexample is given for the case of six full blocks.
We also demonstrate that many known results  \cite[Sec.~9]{Packard1993} can be proved using the proposed SDP approach.

In most applications of structured singular values the considered matrices are complex. Still, several results on real matrices were  presented in \cite{Packard1993}.
It has also been shown that the structured singular value of \emph{nonnegative} real matrices equal the convex upper bound for any number of full or repeated scalar blocks \cite{Colombino2016}.
The investigations in this letter were inspired by the (academically) interesting problem of computing the worst-case contraction factor of the Davis--Yin-splitting operator in $\Reals^2$ \cite{Ryu2018}. 

A rank-constrained SDP formulation similar to the one in this letter (but with larger matrices) was used in \cite{Meinsma1997}.

\emph{Notation:}
We denote the real Hilbert space of symmetric matrices in $\Reals^{r \times r}$ by $\SymMats{r}$ and the real Hilbert space of Hermitian matrices in $\Complexs^{r \times r}$ by $\HermMats{r}$; the standard inner product $\langle X, Y \rangle = \text{trace}(X Y)$ is assumed in both cases.
The positive semidefinite cone in $\SymMats{r}$ is denoted by $\PSDCone{r}$, the positive definite cone is denoted by $\PDCone{r}$, and the subset of $\PSDCone{r}$ with rank $\leq q$ matrices is denoted by $\LowRank_q(\PSDCone{r})$; analogous notation is used in the Hermitian case.
The $n \times n$ identity matrix is denoted by $I_n$. The largest singular value of a matrix $A$ is denoted by $\sigmamax(A)$. 

\section{Background}

\subsection{The Structured Singular Value and an Upper Bound}
Since the focus of this letter is on complex full-block uncertainty, we specialize the background to this case.

\begin{definition}[\cite{Packard1993}]
Let a matrix $M \in \Complexs^{n \times n}$ and $F$ block sizes $n_j$ such that $\sum_{j=1}^F n_j = n$ be given. The \emph{structured singular value} of $M$ with respect to the uncertainty structure
\[
\bDelta = \{\text{diag}(\Delta_1, \ldots, \Delta_F) : \Delta_i \in \Complexs^{n_j \times n_j}
\},
\]
is defined by
\begin{equation*}
\mu_\bDelta (M) \coloneqq
\frac{1}{
	\min \left\{ \sigmamax(\Delta) \,:\, \Delta \in \bDelta, \, \det(I - M\Delta) = 0 \right\} }
\label{eq:mu_definition}
\end{equation*}
unless $\det(I - \Delta M) \neq 0$ for all $\Delta \in \bDelta$, in which case $\mu_\bDelta(M) \coloneqq 0$.	
\end{definition}

With $\bm{D} = \big\{ d_1 I_{n_1}, \ldots , d_F I_{n_F}) \mid d_j \in \Reals, \, d_j>0 \big\}$
the following upper bound can be shown \cite{Packard1993}
\begin{equation}
\mu_\bDelta(M)  \leq \inf_{D \in \bm{D}} \sigmamax(D^{1/2}MD^{-1/2}) \eqqcolon  \nu_\bDelta(M).
\label{eq:nu_definition}
\end{equation}

The upper bound $\nu_\bDelta(M)$ can be computed by convex optimization \cite{Packard1993}.  The structured singular value $\mu_\bDelta(M)$, on the other hand,  is in general NP hard to compute.  However, for $F \leq 3$ it holds that 
$\mu_\bDelta(M) \!=\! \nu_\bDelta(M)$ for any $M \in \Complexs^{n \times n}$ \cite{Packard1993}.
In this letter we show that this equality holds for $F \leq 5$  if $M \in \Reals^{n \times n}$.

\subsection{\boldmath Conditions for $\nu_\bDelta(M) = \sigmamax(M)$}

If $\nu_\bDelta(M) = \sigmamax(M)$ then the infimum in \eqref{eq:nu_definition} is attained for $D = I_n$ and the matrix $M$ is said to be \emph{optimally $D$ scaled}. 
That a matrix $M$ is optimally $D$ scaled is equivalent to that the function $D \mapsto \sigmamax(D^{1/2}MD^{-1/2})$ lacks descent directions in the point $D=I_n$ \cite{Doyle1982}.
This ``lack of descent directions'' can be characterized from a singular value decomposition of $M$ \cite[Sec. 8]{Packard1993}. 
Let a singular value decomposition of $M$ be given by
\begin{equation}\label{eq:svd_of_M}
M = \sigma_1 UV^* + \widetilde{U} \widetilde{\Sigma} \widetilde{V}^*,
\end{equation}
where $U$ and $V$ are $n \times r$ matrices whose columns are the $r$ pairs of singular vectors that correspond to the largest singular value $\sigma_1 = \sigmamax(M)$.
\stepcounter{theorem}%
\edef\thmnbroptscaling{\thetheorem}%
Theorems 8.1 and 8.2 in \cite{Packard1993} can be combined into the following.

\vspace{0.5em}
\begin{manualtheorem}{\thmnbroptscaling}
$\!\nu_\bDelta(M) = \sigmamax(M)$ $\iff$\\
no $Z \in \{\, \text{diag}(z_1 I_{n_1}, \ldots, z_{F-1} I_{n_{F-1}}, 0_{n_F \times n_F}) \mid z_j \in \Reals \}$
satisfies $\lambda_\text{min}(U^*ZU - V^*ZV) > 0$.
\end{manualtheorem}

We will need a more geometric condition than the minimum-eigenvalue condition in Theorem 1. Let $U_j$ and $V_j$ be the $n_j \times r$ matrices that are given by the $n_j$ rows of $U$ and $V$ that correspond to the $j$th uncertainty block, that is
\begin{equation}
U = \bmat{U_1 \\ \vdots \\ U_F} , \qquad V = \bmat{V_1 \\ \vdots \\ V_F}.
\label{eq:UV_decomposition}
\end{equation}
For each (full) uncertainty block, define the Hermitian $r \times r$ matrix
\begin{equation}\label{eq:P_full}
P_j \coloneqq  U_j^* U_j - V_j^* V_j
\end{equation}
and let
\begin{equation}\label{eq:Pset_def}
\Pset \coloneqq \left\{ P_1, \ldots, P_{F-1} \right\}.
\end{equation}
Theorem 1 can now be formulated as follows%
\footnote{Note that, in Theorem 1, $U^* Z U - V^* Z V = \sum_{j=1}^{F-1} z_j P_j$.}\!.

\begin{manualtheorem}{\thmnbroptscaling H}\label{prop:nu_equiv_complex}
	\mbox{$\nu_\bDelta(M) = \sigmamax(M)
			\iff
			\HermPDCone{r} \cap \Span(\Pset)   = \emptyset$.}
\end{manualtheorem}

For our results on real matrices $M$ we need the following result that follows trivially from  Theorem~\ref{prop:nu_equiv_complex}.
\begin{manualtheorem}{\thmnbroptscaling S}\label{prop:nu_equiv_real}
	If all matrices in $\Pset$ are real then\\
	\hspace*{2.9cm}$\nu_\bDelta(M) = \sigmamax(M) 
	\iff
	\PDCone{r} \cap \Span(\Pset) = \emptyset.$
\end{manualtheorem}

\subsection{\boldmath Condition for $\mu_\bDelta(M) = \sigmamax(M)$}

The equality $\mu_\bDelta(M) = \sigmamax(M)$ is equivalent to that a certain system of quadratic equations in the matrices  $P_j$ in \eqref{eq:P_full} has a nontrivial solution
\cite[Thm. 8.3]{Packard1993}.

\begin{theorem}\label{prop:mu_equiv_doyle}
	$\mu_\bDelta(M) = \sigmamax(M)$ $\iff$\\[0.15em]
	there is a nonzero vector $\eta \in \Complexs^r$ such that
	
	\vspace{-0.8em}
	\begin{equation}
	\langle P, \eta\eta^* \rangle = \text{trace} (P \eta\eta^*) = \eta^* P \eta  = 0 \text{\,\, for all\,\,} P \in \Pset.
	\label{eq:trace_condition}
	\end{equation}
\end{theorem}

\begin{remark*}
The condition $\norm{\eta} = 1$ in \cite[Thm. 8.3]{Packard1993} has without loss of generality been relaxed to nonzeroness of $\eta$.
\end{remark*}

\subsection{Low-Rank Solutions to Feasibility SDPs}
The positive results in this letter follow from the following theorem that states:
``Given a low-dimensional subspace $L$ of $\SymMats{r}$ that does not intersect the positive definite cone $\PDCone{r}$, then it is possible to find a nonzero low-rank positive semidefinite matrix orthogonal to $L$''.
Recall that $\LowRank_q(\PSDCone{r})$, with $q \leq r$, denotes the positive semidefinite $r \times r$ matrices of rank $\leq q$.
\stepcounter{theorem}
\edef\thmnbrlowrank{\thetheorem}
\begin{manualtheorem}{\thmnbrlowrank S}\label{thm:low_rank_sym}
	Let $L$ be a linear subspace of $\SymMats{r}$.
	If $\dim L  \leq   (q+1)(q+2)/2 - 2$ and $\PDCone{r} \cap L  = \emptyset$ then $ \LowRank_q(\PSDCone{r}) \cap L^\perp \neq \{0\}$.
\end{manualtheorem}

\begin{proof}
Follows from \cite[Sec.~2.2]{Barvinok1995} (using one constraint to ensure a nonzero solution) or from \cite[Thm. 6]{Hiriart2002}.
\end{proof}

\begin{remark*}
Barvinok's result in \cite{Barvinok1995} is essentially a consequence of the facial structure of the positive semidefinite cone $\PSDCone{r}$ \cite{Hiriart2002}.
Every face of $\PSDCone{r}$ is isomorphic to $\PSDCone{q}$ for some $q \leq r$, which is where the number $(q+1)(q+2)/2 = \dim \, \SymMats{{q+1}}$ in Theorem 1 comes from.
A generalization of Barvinok's result to convex cones in Euclidean spaces, of a form similar to Theorem~1, is given in \cite[Thm. 6]{Hiriart2002}.
\end{remark*}

For one of our results we need the following variation of Theorem~\ref{thm:low_rank_sym} which can be shown as in \cite{Barvinok1995}, or perhaps more directly from \cite[Thm. 6]{Hiriart2002}.
\begin{manualtheorem}{\thmnbrlowrank H}\label{thm:low_rank_herm}
	Let $L$ be a linear subspace of $\HermMats{r}$. If $\dim L \leq  (q+1)^2 -2$ and $  \HermPDCone{r} \cap L = \emptyset$ then $\LowRank_q(\HermPSDCone{r}) \cap L^\perp \neq \{0\}$.
\end{manualtheorem}

\section{New Results}\label{sec:new_results}

\begin{theorem}\label{thm:F5}
	If $M \in \Reals^{n \times n}$ and $F \leq 5$ then $\mu_\bDelta(M) = \nu_\bDelta(M)$.
\end{theorem} 

\begin{proof}
\emph{Part 1:}
We begin by showing that if $M$ is optimally $D$ scaled (i.e.,  $\nu_\bDelta (M) = \sigmamax(M)$) then $\nu_\bDelta(M) = \mu_\bDelta(M)$.

Assume that $M \in \Reals^{n \times n}$ satisfies $\nu_\bDelta (M) = \sigmamax(M)$.
Take a real singular value decomposition \eqref{eq:svd_of_M} of $M$ and let $\Pset$ be the set in \eqref{eq:Pset_def}. Note that the matrices in $\Pset$ are real and that $\dim \Span(\Pset) \leq F - 1 \leq 4$.
Working in $\SymMats{r}$, it follows from Theorem~\ref{prop:nu_equiv_real} and Theorem~\ref{thm:low_rank_sym} that there exists a nonzero $X \in \LowRank_2(\PSDCone{r}) \cap \Span(\Pset)^\perp $. 

\mbox{From $X \! \in \! \LowRank_2(\PSDCone{r})$ we get $X \!= \eta_\re \eta_\re^\tran + \eta_\im \eta_\im^\tran \!=\! \Re\, \eta\eta^*$} where $\eta = \eta_\re + i\eta_\im \in \Complexs^r$ is nonzero since $X$ is nonzero.
From $X \in \Span(\Pset)^\perp$ we get that $\langle P,  \Re\, \eta\eta^* \rangle = 0$ for all $P \in \Pset$.
This implies \eqref{eq:trace_condition} since all $P \in \Pset$ are real and $\langle P,  \eta\eta^* \rangle = \eta^* P \eta$ is \emph{always} real. Theorem~\ref{prop:mu_equiv_doyle} now gives that $\mu_\bDelta(M) = \sigmamax(M)$.

	\emph{Part 2:} Extending Part 1 to any $M \in \Reals^{n \times n}$ can be done as in the proof of \cite[Thm. 8.4]{Packard1993} if also realness is considered. We give an outline and refer to \cite{Packard1993} for details.
	
	Let $M$ be any matrix in $\Reals^{n \times n}$. It can be shown that there exists an optimally $D$ scaled matrix $W$ to which $M$ can be made arbitrarily close through $D$ scaling. Since the factors $D \in \bm{D}$ are real for full-block uncertainty, the matrix $W$ can be assumed to be real. Part 1 now gives that $\nu_\bDelta(W) = \mu_\bDelta(W)$.  Since $\mu_\bDelta(\cdot)$ is invariant under $D$ scaling, and both $\sigmamax(\cdot)$ and $\mu_\bDelta(\cdot)$ are continuous, it follows that $\nu_\bDelta(M) = \mu_\bDelta(M)$.
\end{proof}	

\begin{counterexample}[$M \in \Reals^{n \times n}$, $F=6$]\label{ce:F6}
Let $M = UV^\tran$ where
\[
U =
\frac{1}{2}
\bmat{
	1 & 1 & 0 \\
	1 & -1 & 0 \\
	1 & 0 & 1 \\
	1 & 0 & -1 \\
	0 & 1 & 1 \\
	0 & 1 & -1
},
\quad
V =
\frac{1}{\sqrt{2}}
\bmat{
	0 & 0 & 1 \\
	0 & 0 & 1 \\
	0 & 1 & 0 \\
	0 & 1 & 0 \\
	1 & 0 & 0 \\
	1 & 0 & 0 \\
	
},
\]
and let $\bDelta = \{\text{diag}(\Delta_1, \ldots, \Delta_6) \mid \Delta_j \in \Complexs\}$.

We have
\[
\setlength{\arraycolsep}{2.7pt}
\newcommand{\smin}{\scalebox{0.48}[1.0]{\( - \)}}
\Pset = \!\!
\scalebox{1}{
$
\left\{\!
\dfrac{1}{4}\!\!\bmat{1 & 1 & 0 \\ 1 & 1 & 0 \\ 0 & 0 & \smin2}\!,
\dfrac{1}{4}\!\bmat{1 & \smin1 & 0 \\ \smin1 & 1 & 0 \\ 0 & 0 & \smin2}\!,
\dfrac{1}{4}\!\bmat{1 & 0 & 1 \\ 0 & \smin2 & 0 \\ 1 & 0 & 1}\!,
\dfrac{1}{4}\!\bmat{1 & 0 & \smin1 \\ 0 & \smin2 & 0 \\ \smin1 & 0 & 1}\!,
\dfrac{1}{4}\!\bmat{\smin2 & 0 & 0 \\ 0 & 1 & 1 \\ 0 & 1 & 1}
\!
\right\}\!,
$
}
\]
and working in $\SymMats{r}$ it is easily verified that 
\begin{equation}\label{eq:span_ce_f6}
\Span({\Pset})^\perp = \Span(I_3).
\end{equation}
Since any matrix orthogonal to $I_3$ has diagonal elements that sum to zero, it follows that $\Span(\Pset)$ is disjoint from the positive definite cone $\PDCone{r}$.
Hence by Theorem~\ref{prop:nu_equiv_real} we have that $\nu_\bDelta(M) = \sigmamax(M) = 1$.

Assume that there is a nonzero $\eta = \eta_\re + i\eta_\im$ that satisfies \eqref{eq:trace_condition}. Since all elements of $\Pset$ are real, we then have that $\langle P,\, \Re \{ \eta\eta^* \} \rangle = 0$ for all $P \in \Pset$, or equivalently, that  $\Re \{ \eta\eta^* \} \in \Span(\Pset)^\perp$.
This contradicts \eqref{eq:span_ce_f6} since $\Re \{ \eta\eta^* \} = \eta_\re \eta_\re^\tran + \eta_\im \eta_\im^\tran$ has a rank of at most two.
Hence there is no nonzero $\eta$ satisfying \eqref{eq:trace_condition} and  Theorem~\ref{prop:mu_equiv_doyle} gives that $\mu_\bDelta(M)$ does not equal $\sigmamax(M) = \nu_\bDelta(M)$.\hfill\qed
\end{counterexample}

\section{Alternative Proofs of Known Results\\for Full-Block Uncertainty}
In the next two sections we show that the SDP approach  introduced in Sec.~\ref{sec:new_results} can be used for succinct derivations of several theorems and counterexamples in \cite[Sec.~9]{Packard1993}.

\begin{theorem}[{\cite[Sec 9.2]{Packard1993}}]
	Let $M \in \Complexs^{n \times n}$ and assume $F \leq 3$. Then $\mu_\bDelta(M) = \nu_\bDelta(M)$.
\end{theorem} 

\begin{proof}
	Assume that $\nu_\bDelta (M) = \sigmamax(M)$.
	Let $\Pset$ be the set in \eqref{eq:Pset_def} and note that $\dim \Span(\Pset) \leq F - 1 \leq 2$.
	Working in $\HermMats{r}$, it follows from Theorem~\ref{prop:nu_equiv_complex} and Theorem~\ref{thm:low_rank_herm} that there exists a nonzero $X \in \LowRank_1(\HermPSDCone{r})  \cap \Span(\Pset)^\perp$.  
	Hence there is a nonzero $\eta \in \Complexs^r$ such that $\eta \eta^* = X \in \Span(\Pset)^\perp$. Theorem~\ref{prop:mu_equiv_doyle} now gives that $\mu_\bDelta(M) = \sigmamax(M)$. This shows that $\mu_\bDelta(M) = \nu_\bDelta(M)$ if $M$ is optimally $D$ scaled. The extension to arbitrary $M$ can be done as in the proof of Theorem~\ref{thm:F5} or \cite[Thm. 8.4]{Packard1993}.
\end{proof}	

\begin{counterexample}[$M \in \Complexs^{n \times n}$, $F=4$]\label{ce:F4}
	Consider $M = UV^*$ with
	\[
	U =
	\frac{1}{2}
	\bmat{
		1 & 0 \\
		1 & 1 \\
		1 & i \\
		1 & -1-i
	},
	\quad
	V =
	\frac{1}{2}
	\bmat{
		0 	 &  1\\
		1 	 & -1\\
		1    & -i\\
		1-i & 1
	},
	\]
	and 
	$\bDelta = \{\text{diag}(\Delta_1, \ldots, \Delta_4) \mid \Delta_j \in \Complexs\}$, which essentially is Morton and Doyle's classic counterexample \cite[Sec.~9.3]{Packard1993}.
	Working in $\HermMats{2}$, it can be verified that
	\begin{equation}
	\Span(\Pset)^\perp = \Span(I_2).
	\label{eq:span_morton_and_doyle}
	\end{equation}
	This together with Theorem~\ref{prop:nu_equiv_complex} gives that $\nu_\bDelta(M) = \sigmamax(M)$. From \eqref{eq:span_morton_and_doyle} it also follows that there is no nonzero $\eta \in \Complexs^2$ such that $\eta\eta^* \in \Span(\Pset)^\perp$ and 
	by Theorem~\ref{prop:mu_equiv_doyle} we have that $\mu_\bDelta(M) < \sigmamax(M) = \nu_\bDelta(M)$.\qed
\end{counterexample}

\begin{theorem}[{\cite[Sec.~9.7]{Packard1993}}]
Let $M \in \Reals^{n \times n}$ and assume $F \leq 2$.
Then, the smallest perturbation $\Delta \in \bDelta$ that makes $I - \Delta M$ singular can be taken as real. 
\end{theorem} 

\begin{proof}
Assume that $\nu_\bDelta (M) = \sigmamax(M)$.
Let $\Pset$ be the set defined in \eqref{eq:Pset_def} and note that $\dim \Span(\Pset) \leq F-1 \leq 1$.
Working in $\SymMats{r}$, it follows from Theorem~\ref{prop:nu_equiv_real} and Theorem~\ref{thm:low_rank_sym} that there exists a nonzero $X \in \LowRank_{1}(\PSDCone{r}) \cap \Span(\Pset)^\perp$.
Hence there is a nonzero
$\eta \in \Reals^r$ such that $\eta\eta^\tran \in \Span(\Pset)^\perp$.
This implies that $\eta^\tran P \eta = \langle P, \eta\eta^\tran \rangle = 0$ for all $P \in \Pset$.
As in the proof of Theorem 8.3 in \cite{Packard1993}, a real perturbation $\Delta$ can be constructed from this $\eta$.
\end{proof}

\begin{counterexample}[$M \in \Reals^{n \times n}$,  $F=3$]\label{ce:MRF3}
Consider $M = U^\tran V$ with
\[
U =
\frac{1}{2}
\bmat{
	\sqrt{2} & 0 \\
	1 & \sqrt{2} \\
	1 & -\sqrt{2}
},
\quad
V =
\frac{1}{2}
\bmat{0 &  \sqrt{2} \\
	  \sqrt{2}	 &  -1\\
	  \sqrt{2} 	 & 1\\
},
\]
and $\bDelta = \{\text{diag}(\Delta_1, \Delta_2, \Delta_3) \mid \Delta_j \in \Complexs\}$ \footnote{This is one instance of the counterexamples in
\cite[Sec 9.8]{Packard1993}.
Minor changes to $U$ and $V$ were made for consistency.}.
It is easily verified that any real matrix orthogonal to the matrices  in $\Pset$ is a multiple of $I_2$.
This shows that no nonzero $\eta \in \Reals^2$ satisfies \eqref{eq:trace_condition}, and hence the smallest perturbation that makes $I - M\Delta$ singular cannot be real valued.\hfill\qed
\end{counterexample}

\section{Alternative Counterexamples for Repeated Scalar Uncertainty}
\label{sec:repeated_scalar_blocks}
The approach to the counterexamples in the previous sections can also be used for uncertainty structures $\bDelta$ that in addition to full blocks include \emph{repeated scalar} blocks $\delta_j I_{n_j}$ where $\delta_j \in \Complexs$.
As we will show, this enables a unified treatment of the counterexamples in \cite[Secs. 9.5, 9.6, 9.9]{Packard1993}.
See \cite{Packard1993} for further details on scalar uncertainty.

In the case of repeated scalar uncertainty there are $S+F$ blocks in \eqref{eq:UV_decomposition}, with the first $S$ blocks corresponding to repeated scalar uncertainty.
For each repeated scalar block, we define the following set of $r \times r$ matrices 
\begin{align}
\Pset_j  \coloneqq & \left\{ U_j^* E_{k\ell} U_j - V_j^* E_{k\ell} V_j \mid 1 \leq k \leq \ell \leq r \right\} \notag \\
& \cup \left\{ U_j^* F_{k\ell} U_j - V_j^* F_{k\ell} V_j \mid 1 \leq k < \ell \leq r \right\}
\end{align}
where $E_{k\ell}$ is the $n_j \times n_j$ matrix with ones at positions $\left\{ (k,\ell),\,(\ell, k) \right\}$ and zeros elsewhere, and $F_{k\ell}$ is the $n_j \times n_j$ matrix with $i$ at position $(k,\ell)$, $-i$ at position $(\ell, k)$, and zeros elsewhere. With the matrix $\Pset$ in \eqref{eq:Pset_def} changed to
\begin{equation}
\Pset = \Pset_1 \cup \cdots \cup \Pset_S  \cup \left\{ P_{S+1}, \ldots, P_{S+F-1} \right\}
\label{eq:Pset_alt_def}
\end{equation}
it can be verified\footnote{Note that $\Span(\Pset_j) = \{ U_j^* Z U_j - V_j^* Z V_j \mid Z \in \HermMats{{n_j}} \}$.} that Theorems \ref{prop:nu_equiv_complex} and \ref{prop:mu_equiv_doyle} also hold for repeated scalar blocks with $\Pset$ defined as in \eqref{eq:Pset_alt_def}.

\begin{counterexample}[$S=1$, $F=2$]\label{ce:S1F2}
Consider $M = UV^\tran$ with
\[
U =
\frac{1}{\sqrt{3}}
\bmat{
	1 & -1 \\
	1 & 1 \\
	1 & 0 \\
	0 & 1
},
\quad
V =
\frac{1}{\sqrt{3}}
\bmat{
	1 & 1 \\
	1 & -1 \\
	0 & 1 \\
	1 & 0 
},
\]
and $\bDelta = \{\text{diag}(\delta_1 I_2, \Delta_1, \Delta_2) \mid \delta_1, \Delta_1,\Delta_2 \in \Complexs\}$. It can be verified that $\Span(\Pset)^\perp = \Span(I_2)$ and as in Counterexample~\ref{ce:F4} it follows that $\mu_\bDelta(M) < \nu_\bDelta(M)$.\qed
\end{counterexample}

\begin{counterexample}[$S=2$, $F=0$]
Consider $M$, $U$, and $V$ as in Counterexample~\ref{ce:S1F2} and 
$\bDelta = \{\text{diag}(\delta_1 I_2, \delta_2 I_2) \mid \delta_1, \delta_2 \in \Complexs\}$, we once again get $\Span(\Pset)^\perp \!= \Span(I_2)$, and hence $\mu_\bDelta(M) \!<\! \nu_\bDelta(M)$.\qed
\end{counterexample}

\begin{counterexample}[$M \in \Reals^{n \times n}$, $S=1$, $F=1$]
Consider $M = UV^\tran$ with
\[
U =
\frac{1}{\sqrt{2}}
\bmat{
	1 & -1 \\
	1 & 1 \\
	0 & 0
},
\quad
V =
\frac{1}{\sqrt{2}}
\bmat{
	1 & 1 \\
	-1 & 1 \\
	0 &  0
},
\]
and $\bDelta \!=\! \{\text{diag}(\delta_1 I_2, \Delta_1) \mid \delta_1, \Delta_1 \!\in\! \Complexs\}$. Reasoning as in Counterexample~\ref{ce:MRF3} shows that the smallest perturbation that makes $I - M\Delta$ singular cannot be real valued. \qed
\end{counterexample}

\section{Conclusion}
It has been shown that the structured singular value of a real matrix with respect to five full complex uncertainty blocks equals its convex upper bound. A counterexample was provided in the case of six uncertainty blocks.

\section*{Acknowledgment}
The author thanks Anders Rantzer, Bo Bernhardsson, and the anonymous reviewers for helpful comments and suggestions.
The characterization of rank-two matrices in Sec.~\ref{sec:new_results} was proposed by Mattias Fält.

\printbibliography

\end{document}